\documentclass[12pt]{article}
\usepackage{amsmath}
\usepackage{mathtools}
\DeclarePairedDelimiter{\abs}{\lvert}{\rvert}
\DeclarePairedDelimiter{\pare}{(}{)}
\DeclarePairedDelimiter{\set}{\{}{\}}
\DeclarePairedDelimiter{\norm}{\lVert}{\rVert}

\usepackage{amsfonts}
\usepackage{amsthm}
\usepackage{autobreak}
\usepackage[english]{babel} 
\usepackage{bold-extra}
\usepackage{indentfirst}
\usepackage{hyperref}
\usepackage{enumitem}
\usepackage{todonotes}
\usepackage[margin=1.1in]{geometry}
\newtheorem{definition}{Definition}
\newtheorem{lemma}[definition]{Lemma}
\newtheorem{theorem}[definition]{Theorem}
\newtheoremstyle{r}
{}
{}
{\normalfont}
{}
{\scshape}
{.}
{ }
{}
\theoremstyle{r}
\newtheorem{remark}[definition]{Remark}
\numberwithin{definition}{section}
\numberwithin{equation}{section}
\allowdisplaybreaks
\title{Universal Radial Approximation in Spaces of Analytic Functions}
\author{Konstantinos Maronikolakis}
\date{}
\begin{document}
	\maketitle
	\begin{abstract}
		Recently, Charpentier showed that there exist holomorphic functions $f$ in the\break unit disk such that, for any proper compact subset $K$ of the unit circle,\break any continuous function $\phi$ on $K$ and any compact subset $L$ of the unit disk, there exists an increasing sequence $(r_n)_{n\in\mathbb{N}}\subseteq[0,1)$ converging to 1 such that $\abs*{f\pare*{r_n\pare*{\zeta-z}+z}-\phi(\zeta)}\to0$ as $n\to\infty$ uniformly for $\zeta\in K$ and $z\in L$ (see \cite{Charpentier1}). In this paper, we give analogues of this result for the Hardy spaces $H^p(\mathbb{D}),1\leq p<\infty$. In particular, our main result implies that, if we fix a compact subset $K$ of the unit circle with zero arc length measure, then there exist functions in $H^p(\mathbb{D})$ whose radial limits can approximate every continuous function on $K$. We give similar results for the Bergman and Dirichlet spaces.
	\end{abstract}
	\section{Introduction}
	The study of the boundary behaviour of holomorphic functions in the unit disk $\mathbb{D}$ is of great importance. A vast range of results exist regarding holomorphic functions in $\mathbb{D}$ that have wild boundary behaviour. Throughout the paper, we will say that a subset of a metric space $X$ is residual if it contains a dense $G_\delta$ subset of $X$ and that a property is satisfied generically if the set of elements with this property is residual.
	
	Firstly, we will look at the universal Taylor series established by Nestoridis in 1996, which have been shown to possess interesting boundary properties. We define $U(\mathbb{D},0)$ to be the set of functions holomorphic functions $f$ in $\mathbb{D}$ with the following property: For every compact set $K\subseteq\mathbb{C}\setminus\mathbb{D}$ with connected complement and any function $h$ continuous on $K$ and holomorphic in $K^\mathrm{o}$, there exists a subsequence $\pare*{S_{\lambda_n}(f)}_{n\in\mathbb{N}}$ of the partial sums $\pare*{S_n(f)}_{n\in\mathbb{N}}$ of the Taylor series of $f$ at 0 that converges to $h$ uniformly on $K$.This class of functions is a dense $G_\delta$ subset of $H(\mathbb{D})$, where $H(\mathbb{D})$ is the space of holomorphic functions on the unit disk endowed with the topology of uniform convergence on all compact subsets of $\mathbb{D}$ \cite{Nestoridis1}.
	
	In \cite{Bayart2}, Bayart proved that for any $f\in U(\mathbb{D},0)$, any $\zeta\in\mathbb{T}:=\partial\mathbb{D}$, with at most one exception, and any open disk $D_\zeta\subseteq\mathbb{D}$ with $\overline{D_\zeta}\bigcap\mathbb{T}=\{\zeta\}$, the set $f(D_\zeta)$ is unbounded. In \cite{Gardiner2}, Gardiner and Khavinson showed the following analogue of Picard's Theorem for the class $U(\mathbb{D},0)$: Let $f\in U(\mathbb{D},0)$, then for every $\zeta\in\mathbb{T}$ and $r>0$, the function $f$ takes every value in $\mathbb{C}$, with at most one exception, infinitely often on $D(\zeta,r)\bigcap\mathbb{D}$. The boundary behaviour of the derivatives of functions in $U(\mathbb{D},0)$ has also been studied. More specifically, if $f\in U(\mathbb{D},0)$, then there exists a residual subset $G$ of $\mathbb{T}$, such that the set $\{f^{(n)}(r\zeta):0<r<1\}$ is unbounded for every $\zeta\in G$ and every positive integer $n$ \cite{Armitage1}.
	
	Since the class $U(\mathbb{D},0)$ is a dense $G_\delta$ subset of $H(\mathbb{D})$, the boundary properties mentioned above are satisfied generically in $H(\mathbb{D})$. For more results regarding the boundary behaviour of functions in $U(\mathbb{D},0)$, see \cite{Melas1}, \cite{Costakis2}, \cite{Bernal1}, \cite{Gardiner3} and \cite{Gardiner1}.
	
	Next, we will state some results regarding maximal cluster sets. We give the appropriate definitions. Let $f\in H(\mathbb{D})$ and $A\subseteq\mathbb{D}$ such that $\overline{A}\bigcap\mathbb{T}\neq\emptyset$. Let $\widehat{\mathbb{C}}:=\mathbb{C}\bigcup\{\infty\}$ denote the extended complex plane then the cluster set of $f$ along $A$ is the set
	\begin{align*}
		C_A(f)=\Big\{w\in\widehat{\mathbb{C}}:& \ w=\lim_{n\to\infty}f(z_n)\text{ where }(z_n)_{n\in\mathbb{N}}\text{ is a sequence in }A\\			&\text{ such that }|z_n|\to1\Big\}.
	\end{align*}
	Similarly, if $\zeta_0\in\overline{A}\bigcap\mathbb{T}$, then the cluster set of $f$ along $A$ at $\zeta_0$ is the set
	\begin{align*}
		C_A(f,\zeta_0)=\Big\{w\in\widehat{\mathbb{C}}:& \ w=\lim_{n\to\infty}f(z_n)\text{ where }(z_n)_{n\in\mathbb{N}}\text{ is a sequence in }A\\			&\text{ that converges to }\zeta_0\Big\}.
	\end{align*}
	For background on cluster sets, we refer to \cite{Noshiro1} and \cite{Collingwood1}. We will say that a cluster set is maximal if it is equal to $\widehat{\mathbb{C}}$. In this context, functions in $H(\mathbb{D})$ having universal approximation properties along any continuous path having limit points in $\mathbb{T}$, have been shown to exist. More specifically, in \cite{Bernal2}, the authors constructed a dense linear subspace of $H(\mathbb{D})$, every element of which, except for zero, has maximal cluster set along any curve $\gamma\subseteq\mathbb(D)$ tending to $\mathbb{T}$ with $Osc(\gamma)\neq\mathbb{T}$, where $Osc(\gamma)=\{t\in\mathbb{T}:t=\lim_{n\to\infty}z_n\text{ where }(z_n)_{n\in\mathbb{N}}\text{ is a sequence in }\gamma\text{ that converges to some point of }\mathbb{T}\}$. The same authors showed that the set of $f\in H(\mathbb{D})$ such that $C_{\rho_\zeta}(f^{(l)},\zeta)=\widehat{\mathbb{C}}$ for every $l\in\mathbb{N}$ and $\zeta\in\mathbb{T}$ is residual, where $\rho_\zeta=\{r\zeta:r\in[0,1)\}$ for $\zeta\in\mathbb{T}$ \cite{Bernal3}. For a survey on similar results, we refer to \cite{Prado1}.
	
	Further, we discuss functions in $H(\mathbb{D})$ whose radial limits universally approximate functions defined on $\mathbb{T}$ (or subsets of $\mathbb{T}$). In \cite{Bayart3}, working in several complex variables, Bayart showed that, generically in $H(B)$ (where $B$ is the unit ball of $\mathbb{C}^N$), every function has the following property: Given any measurable function $\phi$ on $\partial B$, there exists an increasing sequence $(r_n)_{n\in\mathbb{N}}\subseteq[0,1)$ converging to 1 such that, for every $z\in B,f\pare*{r_n\pare*{\zeta-z}+z}\to\phi(\zeta)$ as $n\to\infty$ for almost every $\zeta\in\partial B$.
	
	Recently, Charpentier proved the following Theorem (Theorem 1 in \cite{Charpentier1}):
	\begin{theorem}
		We define $\mathcal{V}(\mathbb{D})$ to be the set of all functions $f\in H(\mathbb{D})$ that satisfy the following property: Given any compact subset $K$ of $\mathbb{T}$, different from $\mathbb{T}$, any continuous function $\phi$ on $K$, any compact subset $L$ of $\mathbb{D}$ and any $l\in\mathbb{Z}$ there exists an increasing sequence $(r_n)_{n\in\mathbb{N}}\subseteq[0,1)$ converging to 1 such that
		$$\sup_{\zeta\in K}\sup_{z\in L}\abs*{f^{(l)}\pare*{r_n\pare*{\zeta-z}+z}-\phi(\zeta)}\to0\text{ as }n\to\infty.$$
		Then, $\mathcal{V}(\mathbb{D})$ is a dense $G_\delta$ subset of $H(\mathbb{D})$.
	\end{theorem}
	We note that $f^{(l)}$ denotes the anti-derivative of order $l$ of $f$, for $l<0$ and $f^{(0)}=f$. In \cite{Charpentier1}, a similar result is shown for the the disk algebra $A(\mathbb{D})$. We also refer to \cite{Charpentier3}, which is a paper concerning the universal boundary properties of derivatives of functions in $A(\mathbb{D})$.
	
	In this paper, we give an analogous Theorem for the Hardy space $H^p(\mathbb{D}),1\leq p<\infty$. We recall that any $f\in H^p(\mathbb{D})$ has non-tangential limits $m$-almost everywhere on $\mathbb{T}$, where $m$ is the arc length measure, and so its boundary behaviour is more restricted. Nevertheless, some results exist regarding functions in $H^p(\mathbb{D})$ with wild boundary behaviour.
	
	In \cite{Bernal2}, the following analogue of a previously mentioned result is given for Hardy spaces. Let $1\leq p<\infty$ and $\Gamma$ be a countable collection of curves tending to $\mathbb{T}$, then there exists a dense linear subspace of $H^p(\mathbb{D})$, every element of which, except for zero, has maximal cluster set along any curve $\gamma\in\Gamma$. We also have that, generically in $H^p(\mathbb{D})$, every function $f$ has maximal cluster set $C_{\mathbb{D}}(f,\zeta)$ for every $\zeta\in\mathbb{T}$ \cite{Brown2}.
	
	Our main result regarding the existence of functions in $H^p(\mathbb{D})$ with universal radial limits is Theorem \ref{hardymain}. Moreover, we examine functions in $H^p(\mathbb{D})$ whose derivatives of any order also have universal radial limits (Theorem \ref{hardydermain}). Finally, we give similar results for different subspaces of $H(\mathbb{D})$ such as the Bergman spaces $A^p(\mathbb{D})$ (Theorems \ref{bergmain} and \ref{bergdermain}) and the Dirichlet space $\mathcal{D}$ (Theorem \ref{dirmain}). The main tools for the proofs of our results are simultaneous approximation lemmas given in \cite{Beise1} and \cite{Muller1} (see Lemmas \ref{hardymerg},\ref{bergmerg} and \ref{dirmerg}). The proofs also use methods similar to the proof of Theorem 1 in \cite{Charpentier1} (mentioned above).
	\section{Universal radial limits in spaces of analytic functions}
	\label{pradial}
	Let $1\leq p<\infty$. We denote by $H^p(\mathbb{D})$ the space of $f\in H(\mathbb{D})$ that satisfy
	$$\norm*{f}_{H^p}:=\sup_{0<r<1}\pare*{\frac{1}{2\pi}\int_{0}^{2\pi}\abs*{f\pare*{re^{i\theta}}}^pd\theta}^{\frac{1}{p}}<\infty.$$
	This space, endowed with the norm $\norm{\cdot}_{H^p}$, is a Banach space and so Baire's Theorem is at our disposal. Let $\rho$ be a subset of $[0,1)$ with 1 as a limit point and $K$ a compact subset of $\mathbb{T}$ with $m(K)=0$, where $m$ is the arc length measure.
	\begin{definition}
		We define $\mathcal{V}_{H^p}^K(\mathbb{D},\rho)$ to be the set of $f\in H^p(\mathbb{D})$ that satisfy the following property: Given any continuous function $\phi$ on $K$ and any compact subset $L$ of $\mathbb{D}$, there exists an increasing sequence $(r_n)_{n\in\mathbb{N}}\subseteq\rho$ converging to 1 such that
		$$\sup_{\zeta\in K}\sup_{z\in L}\abs*{f\pare*{r_n\pare*{\zeta-z}+z}-\phi(\zeta)}\to0\text{ as }n\to\infty.$$
	\end{definition}
	\begin{remark}
		We notice that the definition cannot be improved by asking that $m(K)>0$. This is true because any $f\in H^p(\mathbb{D})$ has non-tangential limits $m$-almost everywhere on $\mathbb{T}$ and so it cannot have the universal boundary property described above on a set of positive measure.
	\end{remark}
	\begin{theorem}
		\label{hardymain}
		The set $\mathcal{V}_{H^p}^K(\mathbb{D},\rho)$ is a dense $G_\delta$ subset of $H^p(\mathbb{D})$.
	\end{theorem}
	In order to prove Theorem \ref{hardymain}, it suffices to write $\mathcal{V}_{H^p}^K(\mathbb{D},\rho)$ as a countable intersection of dense open sets and use Baire's Theorem.
	
	Let $\pare*{P_j}_{j\in\mathbb{N}}$ be an enumeration of complex polynomials whose coefficients have rational coordinates. For $j\in\mathbb{N},s\in\mathbb{N}^*$ and $r\in[0,1)$, we introduce the set $$U(j,s,r):=\set*{f\in H^p(\mathbb{D}):\sup_{\zeta\in K}\sup_{z\in \overline{D\pare*{0,\frac{s}{s+1}}}}\abs*{f\pare*{r\pare*{\zeta-z}+z}-P_j(\zeta)}<\frac{1}{s}}.$$
	\begin{lemma}
		\label{hardyint}
		We have that $$\mathcal{V}_{H^p}^K(\mathbb{D},\rho)=\bigcap_{j,s} \ \bigcup_{r\in\rho}U(j,s,r).$$
	\end{lemma}
	\begin{proof}
		The inclusion $\mathcal{V}_{H^p}^K(\mathbb{D},\rho)\subseteq\bigcap_{j,s} \ \bigcup_{r\in\rho}U(j,s,r)$ holds trivially. So, let $f$ belong to the right-hand side. Let $\phi$ be a continuous function on $K,L$ be a compact subset of $\mathbb{D}$ and $\varepsilon>0$. Firstly, we can choose $s$ large enough such that $L\subseteq\overline{D\pare*{0,\frac{s}{s+1}}}$ and $\frac{1}{s}<\frac{\varepsilon}{2}$. By Mergelyan's Theorem, there exists $j\in\mathbb{N}$ such that
		$$\sup_{\zeta\in K}\abs*{P_j(\zeta)-\phi(\zeta)}<\frac{\varepsilon}{2}.$$
		By assumption, there exists $r\in\rho$ such that
		$$\sup_{\zeta\in K}\sup_{z\in \overline{D\pare*{0,\frac{s}{s+1}}}}\abs*{f\pare*{r\pare*{\zeta-z}+z}-P_j(\zeta)}<\frac{1}{s}<\frac{\varepsilon}{2}$$
		and so
		$$\sup_{\zeta\in K}\sup_{z\in L}\abs*{f\pare*{r\pare*{\zeta-z}+z}-\phi(\zeta)}<\varepsilon.$$
		This proves that $f\in\mathcal{V}_{H^p}^K(\mathbb{D},\rho)$.
	\end{proof}
	\begin{lemma}
		\label{hardyopen}
		For any $j\in\mathbb{N},s\in\mathbb{N}^*$ and any $r\in[0,1)$, the set $U(j,s,r)$ is open in $H^p(\mathbb{D})$.
	\end{lemma}
	\begin{proof}
		Let $f\in U(j,s,r)$. Let also $K_{r,s}$ be the set
		$$K_{r,s}=\set*{r(\zeta-z)+z:\zeta\in K,z\in\overline{D\pare*{0,\frac{s}{s+1}}}}.$$
		Clearly, the set $K_{r,s}$ is a compact subset of $\mathbb{D}$, so there exists $C_{r,s,p}>0$ such that $|\phi(z)|\leq C_{r,s,p}\norm*{\phi}_{H^p}$ for every $\phi\in H^p(\mathbb{D})$ and $z\in K_{r,s}$ (see \cite{Duren1} p.36). We set $$\varepsilon:=\frac{\frac{1}{s}-\sup_{\zeta\in K}\sup_{z\in \overline{D\pare*{0,\frac{s}{s+1}}}}\abs*{f\pare*{r\pare*{\zeta-z}+z}-P_j(\zeta)}}{2C_{r,s,p}}>0.$$ It suffices to show that any function $g\in H^p(\mathbb{D})$ with $\norm*{f-g}_{H^p}<\varepsilon$ belongs to $U(j,s,r)$. Indeed, we get that
		\begin{align*}
			\begin{autobreak}
				\MoveEqLeft[0]
				\sup_{\zeta\in K}\sup_{z\in \overline{D\pare*{0,\frac{s}{s+1}}}}\abs*{g\pare*{r\pare*{\zeta-z}+z}-P_j(\zeta)}
				=\sup_{\zeta\in K}\sup_{z\in \overline{D\pare*{0,\frac{s}{s+1}}}}\abs*{g\pare*{r\pare*{\zeta-z}+z}-f\pare*{r\pare*{\zeta-z}+z}+f\pare*{r\pare*{\zeta-z}+z}-P_j(\zeta)}
				\leq\sup_{w\in K_{r,s}}\abs{g(w)-f(w)}
				+\sup_{\zeta\in K}\sup_{z\in \overline{D\pare*{0,\frac{s}{s+1}}}}\abs*{f\pare*{r\pare*{\zeta-z}+z}-P_j(\zeta)}<\frac{1}{s}
			\end{autobreak}
		\end{align*}
		by the definition of $\varepsilon$ and the proof is completed.
	\end{proof}
	The following simultaneous approximation lemma is part of the proof of Theorem 1.1 in \cite{Beise1}.
	\begin{lemma}
		\label{hardymerg}
		Let $p$ and $K$ be as above, then, for every $g\in H^p(\mathbb{D})$, any continuous function $\phi$ on $K$ and any $\varepsilon>0$, there exists a polynomial $P$ such that
		\begin{equation*}
			\norm*{P-g}_{H^p}<\varepsilon\text{\ \ \ and \ \ }\sup_{\zeta\in K}\abs*{P\pare*{\zeta}-\phi(\zeta)}<\varepsilon.
		\end{equation*}
	\end{lemma}
	\begin{lemma}
		\label{hardydense}
		For any $j\in\mathbb{N}$ and $s\in\mathbb{N}^*$, the set $\bigcup_{r\in\rho}U(j,s,r)$ is dense in $H^p(\mathbb{D})$.
	\end{lemma}
	\begin{proof}
		We fix $\varepsilon>0$ and $g\in H^p(\mathbb{D})$. By Lemma \ref{hardymerg}, there exists a polynomial $P$ such that
		\begin{equation*}
			\norm*{P-g}_{H^p}<\varepsilon\text{\ \ \ and \ \ }\sup_{\zeta\in K}\abs*{P\pare*{\zeta}-P_j(\zeta)}<\frac{1}{2s}.
		\end{equation*}
		For any $r\in\rho$, let $K_{r,s}$ be as in the proof of \ref{hardyopen}. By uniform continuity of $P$ on compact sets, there exists $\delta>0$ such that $\abs{P(w_1)-P(w_2)}<\frac{1}{4s}$ whenever $|w_1-w_2|<d,z,w\in\overline{\mathbb{D}}$. For any $\zeta\in K$ and $z\in \overline{D\pare*{0,\frac{s}{s+1}}}$ we have that $\abs*{r\pare*{\zeta-z}+z-\zeta}=(1-r)\abs*{\zeta-z}\leq2(1-r)$. Using the fact that 1 is a limit point of $\rho$, there exists $r\in\rho$ such that $2(1-r)<\delta$ and so
		\begin{equation*}
			\sup_{\zeta\in K}\sup_{z\in \overline{D\pare*{0,\frac{s}{s+1}}}}\abs*{P\pare*{r\pare*{\zeta-z}+z}-P(\zeta)}\leq\frac{1}{4s}<\frac{1}{2s}.
		\end{equation*}
		Then, we get that
		\begin{align*}
			\begin{autobreak}
				\MoveEqLeft[0]
				\sup_{\zeta\in K}\sup_{z\in \overline{D\pare*{0,\frac{s}{s+1}}}}\abs*{P\pare*{r\pare*{\zeta-z}+z}-P_j(\zeta)}
				=\sup_{\zeta\in K}\sup_{z\in \overline{D\pare*{0,\frac{s}{s+1}}}}\abs*{P\pare*{r\pare*{\zeta-z}+z}-P\pare*{\zeta}+P\pare*{\zeta}-P_j(\zeta)}
				\leq\sup_{\zeta\in K}\sup_{z\in \overline{D\pare*{0,\frac{s}{s+1}}}}\abs*{P\pare*{r\pare*{\zeta-z}+z}-P(\zeta)}
				+\sup_{\zeta\in K}\abs*{P\pare*{\zeta}-P_j(\zeta)}
				<\frac{1}{s}
			\end{autobreak}
		\end{align*}
		and the proof is completed.
	\end{proof}
	The proof of Theorem \ref{hardymain} follows trivially by combining Lemmas \ref{hardyint},\ref{hardyopen} and \ref{hardydense} with Baire's Theorem.
	
	We now focus on Bergman Spaces. Let $1\leq p<\infty$. We denote by $A^p(\mathbb{D})$ the space of $f\in H(\mathbb{D})$ that satisfy
	$$\norm*{f}_{A^p}:=\pare*{\int_{\mathbb{D}}\abs*{f}^pd m_2}^{\frac{1}{p}}<\infty$$
	where $m_2$ denotes the normalised Lebesgue measure on $\mathbb{D}$. This space, endowed with the norm $\norm{\cdot}_{A^p}$, is a Banach space and so Baire's Theorem is at our disposal. In this case, the universal approximation is accomplished on more general compact subsets of $\mathbb{T}$ than the $H^p(\mathbb{D})$ case. More specifically, let $E$ be a compact subset of $\mathbb{T}$ such that $m(E)>0$. Then, we say that $E$ satisfies Carleson's condition if 
	$$\ell(E):=\sum_k m\pare*{B_k}\log\pare*{\frac{1}{m\pare*{B_k}}}<\infty$$
	where $\mathbb{T}\setminus E=\bigcup_k B_k$ and the sets $B_k$ are disjoint open subarcs of $\mathbb{T}$. Now, let $K$ be a compact subset of $\mathbb{T}$ such that either $m(K)>0$ and $K$ does not contain a compact subset of positive measure satisfying Carleson's condition or $m(K)=0$ and $\rho$ be a subset of $[0,1)$ with 1 as a limit point. Then, we have the following:
	\begin{definition}
		We define $\mathcal{V}_{A^p}^K(\mathbb{D},\rho)$ to be the set of $f\in A^p(\mathbb{D})$ that satisfy the following property: Given any continuous function $\phi$ on $K$ and any compact subset $L$ of $\mathbb{D}$, there exists an increasing sequence $(r_n)_{n\in\mathbb{N}}\subseteq\rho$ converging to 1 such that
		$$\sup_{\zeta\in K}\sup_{z\in L}\abs*{f\pare*{r_n\pare*{\zeta-z}+z}-\phi(\zeta)}\to0\text{ as }n\to\infty.$$
	\end{definition}
	\begin{theorem}
		\label{bergmain}
		The set $\mathcal{V}_{A^p}^K(\mathbb{D},\rho)$ is a dense $G_\delta$ subset of $A^p(\mathbb{D})$.
	\end{theorem}
	\begin{proof}
		The proof of \ref{bergmain} is similar to the proof of \ref{hardymain} and so we will only give a sketch of it.
		
		Let $\pare*{P_j}_{j\in\mathbb{N}}$ be as above, then we have that
		$$\mathcal{V}_{A^p}^K(\mathbb{D},\rho)=\bigcap_{j,s} \ \bigcup_{r\in\rho}V(j,s,r)$$
		where
		$$V(j,s,r):=\set*{f\in A^p(\mathbb{D}):\sup_{\zeta\in K}\sup_{z\in \overline{D\pare*{0,\frac{s}{s+1}}}}\abs*{f\pare*{r\pare*{\zeta-z}+z}-P_j(\zeta)}<\frac{1}{s}}$$
		for any $j\in\mathbb{N},s\in\mathbb{N}^*$ and $r\in[0,1)$. Using the fact that, for every compact subset $L$ of $\mathbb{D}$, there exists a constant $C_{L,p}$ such that $|\phi(z)|\leq C_{L,p}\norm*{\phi}_{A^p}$ for every $\phi\in A^p(\mathbb{D})$ and $z\in L$ (see \cite{Duren2} p.7, Theorem 1), we can see that the set $V(j,s,r)$ is open in $A^p(\mathbb{D})$ for any $j\in\mathbb{N},s\in\mathbb{N}^*$ and $r\in[0,1)$. It remains to show that the set $\bigcup_{r\in\rho}V(j,s,r)$ is dense in $A^p(\mathbb{D})$ for any $j\in\mathbb{N}$ and $s\in\mathbb{N}^*$. Then, we would have the desired result by Baire's Theorem. Indeed, we can repeat the proof of \ref{hardydense} by replacing the simultaneous approximation lemma with the following analogous result (see proof of Theorem 2.2 in \cite{Beise1}):
		\begin{lemma}
			\label{bergmerg}
			Let $p$ and $K$ be as above, then, for every $g\in A^p(\mathbb{D})$, any continuous function $\phi$ on $K$ and any $\varepsilon>0$, there exists a polynomial $P$ such that
			\begin{equation*}
				\norm*{P-g}_{A^p}<\varepsilon\text{\ \ \ and \ \ }\sup_{\zeta\in K}\abs*{P\pare*{\zeta}-\phi(\zeta)}<\varepsilon.
			\end{equation*}
		\end{lemma}
	\end{proof}
	We now consider the Dirichlet Space. We denote by $D$ the space of $f\in H(\mathbb{D})$ that satisfy
	$$\mathcal{D}(f):=\pare*{\int_{\mathbb{D}}\abs*{f'}^2d m_2}^{\frac{1}{2}}<\infty$$
	We notice that $\mathcal{D}(\cdot)$ is not a norm because $\mathcal{D}(f)=0$ for any constant function $f\in H(\mathbb{D})$. So, we define $\norm*{f}_{D}^2:=\norm*{f}_{H^p}^2+\mathcal{D}(f)^2,f\in D$ which is a norm. The space $D$, endowed with the norm $\norm{\cdot}_{D}$, is a Banach space and so Baire's Theorem is at our disposal. Now, let $K$ be a compact and polar subset of $\mathbb{T}$ and $\rho$ be a subset of $[0,1)$ with 1 as a limit point. Then, we have the following:
	\begin{definition}
		We define $\mathcal{V}_{D}^K(\mathbb{D},\rho)$ to be the set of $f\in D$ that satisfy the following property: Given any continuous function $\phi$ on $K$ and any compact subset $L$ of $\mathbb{D}$, there exists an increasing sequence $(r_n)_{n\in\mathbb{N}}\subseteq\rho$ converging to 1 such that
		$$\sup_{\zeta\in K}\sup_{z\in L}\abs*{f\pare*{r_n\pare*{\zeta-z}+z}-\phi(\zeta)}\to0\text{ as }n\to\infty.$$
	\end{definition}
	\begin{remark}
		We notice that the definition cannot be improved by asking that the set $K$ is non-polar. This is true because, according to Beurling's Theorem (see \cite{El-Fallah1} Theorem 3.2.1), any $f\in D$ has non-tangential limits everywhere except a polar set and so it cannot have the universal boundary property described above on a non-polar set.
	\end{remark}
	\begin{theorem}
		\label{dirmain}
		The set $\mathcal{V}_{D}^K(\mathbb{D},\rho)$ is a dense $G_\delta$ subset of $D$.
	\end{theorem}
	The proof of \ref{dirmain} is similar to the proof of \ref{hardymain}. It is based on the following simultaneous approximation lemma (see \cite{Muller1} Theorem 5):
	\begin{lemma}
		\label{dirmerg}
		Let $K$ be as above, then, for every $g\in D(\mathbb{D})$, any continuous function $\phi$ on $K$ and any $\varepsilon>0$, there exists a polynomial $P$ such that
		\begin{equation*}
			\norm*{P-g}_{D}<\varepsilon\text{\ \ \ and \ \ }\sup_{\zeta\in K}\abs*{P\pare*{\zeta}-\phi(\zeta)}<\varepsilon.
		\end{equation*}
	\end{lemma}
	In \cite{Muller1}, a different norm for $D$ is considered, but the two norms are equivalent so the Lemma still applies.
	\section{Universal radial limits of derivatives}
	\label{pderradial}
	Let $1\leq p<\infty$. Let also $\rho$ be a subset of $[0,1)$ with 1 as a limit point and $K$ be a compact subset of $\mathbb{T}$ with $m(K)=0$.
	\begin{definition}
		We define $\mathcal{DV}_{H^p}^K(\mathbb{D},\rho)$ to be the set of $f\in H^p(\mathbb{D})$ that satisfy the following property: Given any continuous function $\phi$ on $K$, any compact subset $L$ of $\mathbb{D}$ and any $l\in\mathbb{N}$, there exists an increasing sequence $(r_n)_{n\in\mathbb{N}}\subseteq\rho$ converging to 1 such that
		$$\sup_{\zeta\in K}\sup_{z\in L}\abs*{f^{(l)}\pare*{r_n\pare*{\zeta-z}+z}-\phi(\zeta)}\to0\text{ as }n\to\infty.$$
	\end{definition}
	\begin{theorem}
		\label{hardydermain}
		The set $\mathcal{DV}_{H^p}^K(\mathbb{D},\rho)$ is a dense $G_\delta$ subset of $H^p(\mathbb{D})$.
	\end{theorem}
	In order to prove Theorem \ref{hardydermain}, it suffices to write $\mathcal{DV}_{H^p}^K(\mathbb{D},\rho)$ as a countable intersection of dense open sets and use Baire's Theorem.
	
	Let $\pare*{P_j}_{n\in\mathbb{N}}$ be an enumeration of complex polynomials whose coefficients have rational coordinates. For $j,l\in\mathbb{N},s\in\mathbb{N}^*$ and $r\in[0,1)$, we introduce the set $$U(j,s,l,r):=\set*{f\in H^p(\mathbb{D}):\sup_{\zeta\in K}\sup_{z\in \overline{D\pare*{0,\frac{s}{s+1}}}}\abs*{f^{(l)}\pare*{r\pare*{\zeta-z}+z}-P_j(\zeta)}<\frac{1}{s}}.$$
	\begin{lemma}
		\label{hardyderint}
		We have that $$\mathcal{DV}_{H^p}^K(\mathbb{D},\rho)=\bigcap_{j,s,l} \ \bigcup_{r\in\rho}U(j,s,l,r).$$
	\end{lemma}
	The proof is similar to the proof of Lemma $\ref{hardyint}$ and is omitted.
	\begin{lemma}
		\label{hardyderopen}
		For any $j,l\in\mathbb{N},s\in\mathbb{N}^*$ and any $r\in[0,1)$, the set $U(j,s,l,r)$ is open in $H^p(\mathbb{D})$.
	\end{lemma}
	\begin{proof}
		Let $f\in U(j,s,l,r)$. Let $K_{r,s}$ be as in the proof of \ref{hardyopen}.
		Clearly, the set $K_{r,s}$ is a compact subset of $\mathbb{D}$, so, by the Cauchy estimates and the Lemma on p.36 of \cite{Duren1}, we can find $C_{r,s,l,p}>0$ such that $|\phi^{(l)}(z)|\leq C_{r,s,l,p}\norm*{\phi}_{H^p}$ for every $\phi\in H^p(\mathbb{D})$ and $z\in K_{r,s}$. We set $$\varepsilon:=\frac{\frac{1}{s}-\sup_{\zeta\in K}\sup_{z\in \overline{D\pare*{0,\frac{s}{s+1}}}}\abs*{f^{(l)}\pare*{r\pare*{\zeta-z}+z}-P_j(\zeta)}}{2C_{r,s,l,p}}>0.$$ It suffices to show that any function $g\in H^p(\mathbb{D})$ with $\norm*{f-g}_{H^p}<\varepsilon$ belongs to $U(j,s,l,r)$. Indeed, we get that
		\begin{align*}
			\begin{autobreak}
				\MoveEqLeft[0]
				\sup_{\zeta\in K}\sup_{z\in \overline{D\pare*{0,\frac{s}{s+1}}}}\abs*{g^{(l)}\pare*{r\pare*{\zeta-z}+z}-P_j(\zeta)}
				=\sup_{\zeta\in K}\sup_{z\in \overline{D\pare*{0,\frac{s}{s+1}}}}\abs*{g^{(l)}\pare*{r\pare*{\zeta-z}+z}-f^{(l)}\pare*{r\pare*{\zeta-z}+z}+f^{(l)}\pare*{r\pare*{\zeta-z}+z}-P_j(\zeta)}
				\leq\sup_{w\in K_{r,s}}\abs{g^{(l)}(w)-f^{(l)}(w)}
				+\sup_{\zeta\in K}\sup_{z\in \overline{D\pare*{0,\frac{s}{s+1}}}}\abs*{f^{(l)}\pare*{r\pare*{\zeta-z}+z}-P_j(\zeta)}<\frac{1}{s}
			\end{autobreak}
		\end{align*}
		by the definition of $\varepsilon$ and the proof is completed.
	\end{proof}
	We will need the following simultaneous approximation Lemma.
	\begin{lemma}
		\label{hardydermerg}
		Let $p$ and $K$ be as above and $l\in\mathbb{N}$, then, for any continuous function $\phi$ on $K$ and any $\varepsilon>0$, there exists a polynomial $P$ such that
		\begin{equation*}
			\norm*{P}_{H^p}<\varepsilon\text{\ \ \ and \ \ }\sup_{\zeta\in K}\abs*{P^{(l)}\pare*{\zeta}-\phi(\zeta)}<\varepsilon.
		\end{equation*}
	\end{lemma}
	\begin{proof}
		By \ref{hardymerg}, there exists a polynomial $\tilde{P}$ such that
		\begin{equation*}
			\norm{\tilde{P}}_{H^p}<\frac{\varepsilon}{\pi^l}\text{\ \ \ and \ \ }\sup_{\zeta\in K}\abs*{\tilde{P}\pare*{\zeta}-\phi(\zeta)}<\varepsilon.
		\end{equation*}
		Let $P$ be the unique polynomial such that $P^{(l)}=\tilde{P}$ and $P$ has a zero of order at least $l$ at 0. By applying Theorem 3 of \cite{Pritsker1} $l$ times, we get that $\norm*{P}_{H^p}\leq \pi^l\norm*{P^{(l)}}_{H^p}=\pi^l\norm{\tilde{P}}_{H^p}<\varepsilon$ and the proof is completed.
	\end{proof}
	\begin{lemma}
		\label{hardyderdense1}
		For any $j,l\in\mathbb{N}$ and $s\in\mathbb{N}^*$, the set $\bigcup_{r\in\rho}U(j,s,l,r)$ intersects any neighbourhood of 0 in $H^p(\mathbb{D})$.
	\end{lemma}
	\begin{proof}
		We fix $\varepsilon>0$ and $g\in H^p(\mathbb{D})$. By Lemma \ref{hardydermerg}, there exists a polynomial $P$ such that
		\begin{equation*}
			\norm*{P}_{H^p}<\varepsilon\text{\ \ \ and \ \ }\sup_{\zeta\in K}\abs*{P^{(l)}\pare*{\zeta}-P_j(\zeta)}<\frac{1}{2s}.
		\end{equation*}
		For any $r\in\rho$, let $K_{r,s}$ be as in the proof of \ref{hardyopen}. By uniform continuity of $P^{(l)}$ on compact sets, there exists $\delta>0$ such that $\abs{P^{(l)}(w_1)-P^{(l)}(w_2)}<\frac{1}{4s}$ whenever $|w_1-w_2|<d,z,w\in\overline{\mathbb{D}}$. For any $\zeta\in K$ and $z\in \overline{D\pare*{0,\frac{s}{s+1}}}$ we have that $\abs*{r\pare*{\zeta-z}+z-\zeta}=(1-r)\abs*{\zeta-z}\leq2(1-r)$. Using the fact that 1 is a limit point of $\rho$, there exists $r\in\rho$ such that $2(1-r)<\delta$ and so
		\begin{equation*}
			\sup_{\zeta\in K}\sup_{z\in \overline{D\pare*{0,\frac{s}{s+1}}}}\abs*{P^{(l)}\pare*{r\pare*{\zeta-z}+z}-P^{(l)}(\zeta)}\leq\frac{1}{4s}<\frac{1}{2s}.
		\end{equation*}
		Then, we get that
		\begin{align*}
			\begin{autobreak}
				\MoveEqLeft[0]
				\sup_{\zeta\in K}\sup_{z\in \overline{D\pare*{0,\frac{s}{s+1}}}}\abs*{P^{(l)}\pare*{r\pare*{\zeta-z}+z}-P_j(\zeta)}
				=\sup_{\zeta\in K}\sup_{z\in \overline{D\pare*{0,\frac{s}{s+1}}}}\abs*{P^{(l)}\pare*{r\pare*{\zeta-z}+z}-P^{(l)}\pare*{\zeta}+P^{(l)}\pare*{\zeta}-P_j(\zeta)}
				\leq\sup_{\zeta\in K}\sup_{z\in \overline{D\pare*{0,\frac{s}{s+1}}}}\abs*{P^{(l)}\pare*{r\pare*{\zeta-z}+z}-P^{(l)}(\zeta)}
				+\sup_{\zeta\in K}\abs*{P^{(l)}\pare*{\zeta}-P_j(\zeta)}
				<\frac{1}{s}
			\end{autobreak}
		\end{align*}
		and the proof is completed.
	\end{proof}
	\begin{lemma}
		\label{hardyderdense2}
		For any $j,l\in\mathbb{N}$ and $s\in\mathbb{N}^*$, the set $\bigcup_{r\in\rho}U(j,s,l,r)$ is dense in $H^p(\mathbb{D})$.
	\end{lemma}
	\begin{proof}
		We fix $\varepsilon>0$ and $g\in H^p(\mathbb{D})$. Since the polynomials are dense in $H^p(\mathbb{D})$, there exists $j_1\in\mathbb{N}$ such that $\norm*{P_{j_1}-g}_{H^p}<\frac{\varepsilon}{2}$.
		Then, there exists $r_0$ such that
		\begin{equation*}
			\sup_{\zeta\in K}\sup_{z\in \overline{D\pare*{0,\frac{s}{s+1}}}}\abs*{P_{j_1}^{(l)}\pare*{r\pare*{\zeta-z}+z}-P_{j_1}^{(l)}(\zeta)}<\frac{1}{2s}.
		\end{equation*}
		for any $r\geq r_0$. Then, $P_j-P_{j_1}^{(l)}$ is a polynomial with coefficients with rational coordinates, so it is equal to $P_{j_2}$ for some $j_2\in\mathbb{N}$. By \ref{hardyderdense1}, there exists a $r\in\rho$ such that $r\geq r_0$ and $f_0\in U(j_2,2s,l,r)$ such that $\norm*{f_0}_{H^p}<\frac{\varepsilon}{2}$. So, we have that
		$$\sup_{\zeta\in K}\sup_{z\in \overline{D\pare*{0,\frac{s}{s+1}}}}\abs*{f_0^{(l)}\pare*{r\pare*{\zeta-z}+z}-P_{j_2}(\zeta)}<\frac{1}{2s}.$$
		By setting $P:=f_0+P_{j_1}$, we get that $\norm*{P-g}_{H^p}<\varepsilon$, and also
		\begin{align*}
			\begin{autobreak}
				\MoveEqLeft[0]
				\sup_{\zeta\in K}\sup_{z\in \overline{D\pare*{0,\frac{s}{s+1}}}}\abs*{P^{(l)}\pare*{r\pare*{\zeta-z}+z}-P_j(\zeta)}
				=\sup_{\zeta\in K}\sup_{z\in \overline{D\pare*{0,\frac{s}{s+1}}}}\abs*{f_0^{(l)}\pare*{r\pare*{\zeta-z}+z}+P_{j_1}^{(l)}\pare*{r\pare*{\zeta-z}+z}-P_{j_1}^{(l)}(\zeta)-P_{j_2}(\zeta)}
				\leq\sup_{\zeta\in K}\sup_{z\in \overline{D\pare*{0,\frac{s}{s+1}}}}\abs*{f_0^{(l)}\pare*{r\pare*{\zeta-z}+z}-P_{j_2}(\zeta)}
				+\sup_{\zeta\in K}\sup_{z\in \overline{D\pare*{0,\frac{s}{s+1}}}}\abs*{P_{j_1}^{(l)}\pare*{r\pare*{\zeta-z}+z}-P_{j_1}^{(l)}(\zeta)}
				\leq\frac{1}{2s}+\frac{1}{2s}=\frac{1}{s}
			\end{autobreak}
		\end{align*}
		and so $P\in U(j,s,l,r)$.
	\end{proof}
	The proof of Theorem \ref{hardydermain} follows trivially by combining Lemmas \ref{hardyderint},\ref{hardyderopen} and \ref{hardyderdense2} with Baire's Theorem.
	
	We now focus on the Bergman Spaces. Let $2<p<\infty$. Let also $\rho$ be a subset of $[0,1)$ with 1 as a limit point and $K$ be a compact subset of $\mathbb{T}$ such that either $m(K)>0$ and $K$ does not contain a compact subset of positive measure satisfying Carleson's condition or $m(K)=0$.
	\begin{definition}
		We define $\mathcal{DV}_{A^p}^K(\mathbb{D},\rho)$ to be the set of $f\in A^p(\mathbb{D})$ that satisfy the following property: Given any continuous function $\phi$ on $K$, any compact subset $L$ of $\mathbb{D}$ and any $l\in\mathbb{N}$, there exists an increasing sequence $(r_n)_{n\in\mathbb{N}}\subseteq\rho$ converging to 1 such that
		$$\sup_{\zeta\in K}\sup_{z\in L}\abs*{f^{(l)}\pare*{r_n\pare*{\zeta-z}+z}-\phi(\zeta)}\to0\text{ as }n\to\infty.$$
	\end{definition}
	\begin{theorem}
		\label{bergdermain}
		The set $\mathcal{DV}_{A^p}^K(\mathbb{D},\rho)$ is a dense $G_\delta$ subset of $A^p(\mathbb{D})$.
	\end{theorem}
	The proof of Theorem \ref{bergdermain} is similar to the proofs of \ref{bergmain} and \ref{hardydermain} and is omitted. The main tool is the following Lemma:
	\begin{lemma}
		\label{bergdermerg}
		Let $p$ and $K$ be as above and $l\in\mathbb{N}$, then, for any continuous function $\phi$ on $K$ and any $\varepsilon>0$, there exists a polynomial $P$ such that
		\begin{equation*}
			\norm*{P}_{A^p}<\varepsilon\text{\ \ \ and \ \ }\sup_{\zeta\in K}\abs*{P^{(l)}\pare*{\zeta}-\phi(\zeta)}<\varepsilon.
		\end{equation*}
	\end{lemma}
	The proof is similar to the proof of \ref{hardydermerg}. The main tool of the proof is Theorem 9 (13) of \cite{Pritsker1}. We notice that the inequality in \cite{Pritsker1} only holds for $p>2$ and so this method can not be applied to prove a similar result for $1\leq p\leq2$.
	\vspace{10mm}
	\par
	\textit{Acknowledgement}. I would like to thank Dr Myrto Manolaki for her guidance throughout the creation of the paper and Professors Vassili Nestoridis and Paul Gauthier for their helpful suggestions. I would also like to acknowledge financial support from the Irish Research Council.
	\bibliographystyle{amsplain}
	\bibliography{refs}
	\itshape
	School of Mathematics and Statistics\\
	University College Dublin\\
	Belfield, Dublin 4, Ireland\\[\baselineskip]
	E-mail:\\
	conmaron@gmail.com
\end{document}